\documentclass[11pt,twoside,leqno]{article}
\usepackage[all]{xy}
\usepackage{amsfonts,amssymb,amscd,amsmath,amsthm,enumerate,verbatim,calc,latexsym}
\usepackage[colorlinks=true,linkcolor=blue,urlcolor=blue,citecolor=blue,anchorcolor=blue,bookmarksnumbered=blue]{hyperref}
\newtheorem{thm}{Theorem}[section]

\newtheorem{cor}[thm]{Corollary}

\newtheorem{ques}[thm]{Question}

\numberwithin{equation}{section}
\begin{document}
%%% ----------------------------------------------------------------------
%%% ----------------------------------------------------------------------
%%% ----------------------------------------------------------------------
%%%
%%%
%%%
\title{\bf Cofiniteness and finiteness of associated prime ideals of generalized local cohomology modules}
%%%
%%%
%%%
%%% ----------------------------------------------------------------------
%%% ----------------------------------------------------------------------
%%% ----------------------------------------------------------------------
\author{Alireza Vahidi, Ahmad Khaksari, and Mohammad Shirazipour}

\date{}
\maketitle

\renewcommand{\thefootnote}{}
\footnote{This research of Alireza Vahidi was in part supported by a grant from Payame Noor University.}
%\footnote{$^\ast$Corresponding author.}
\renewcommand{\thefootnote}{\arabic{footnote}}
\setcounter{footnote}{0}
%%% ----------------------------------------------------------------------
%%% ----------------------------------------------------------------------
%%% ----------------------------------------------------------------------
%%%
%%%
%%%
%%% Abstract
%%%
%%%
%%%
%%% ----------------------------------------------------------------------
%%% ----------------------------------------------------------------------
%%% ----------------------------------------------------------------------
\begin{abstract}
Let $n$ be a non-negative integer, $R$ a commutative Noetherian ring, $\mathfrak{a}$ an ideal of $R$, $M$ and $N$ two finitely generated $R$-modules, and $X$ an arbitrary $R$-module. In this paper, we study cofiniteness and finiteness of associated prime ideals of generalized local cohomology modules. In some cases, we show that $\operatorname{H}^{i}_{\mathfrak{a}}(M,X)$ is an $(\operatorname{FD}_{<n},\mathfrak{a})$-cofinite $R$-module and $\{\mathfrak{p}\in\operatorname{Ass}_R(\operatorname{H}^{i}_{\mathfrak{a}}(M,X)):\dim(R/\mathfrak{p})\geq{n}\}$ is a finite set for all $i$. If $R$ is semi-local, we observe that $\operatorname{Ass}_R(\operatorname{H}^{i}_{\mathfrak{a}}(M,N))$ is finite for all $i$ when $\dim_R(M)\leq{3}$ or $\dim_R(N)\leq{3}$. Also, in some situations, we prove that $\operatorname{H}^{i}_{\mathfrak{a}}(M,X)$ is an $\mathfrak{a}$-cofinite $R$-module for all $i$.
\end{abstract}
%%% ----------------------------------------------------------------------
%%% ----------------------------------------------------------------------
%%% ----------------------------------------------------------------------
%%%
%%%
%%%
\section{Introduction}\label{1}
%%%
%%%
%%%
%%% ----------------------------------------------------------------------
%%% ----------------------------------------------------------------------
%%% ----------------------------------------------------------------------
Throughout, let $R$ denote a commutative Noetherian ring with non-zero identity, $\mathfrak{a}$ an ideal of $R$, $M$ and $N$ two finite (i.e., finitely generated) $R$-modules, $X$ and $Y$ two arbitrary $R$-modules which are not necessarily finite, and $n$ and $t$ two non-negative integers. For basic results, notations, and terminology not given in this paper, readers are referred to \cite{BSh, BH, Rot}.

It is well known that the local cohomology module $\operatorname{H}^{i}_{\mathfrak{m}}(N)$ is an Artinian $R$-module and so $\operatorname{Hom}_R(R/\mathfrak{m}, \operatorname{H}^i_\mathfrak{m}(N))$ is a finite $R$-module for all $i$ when $R$ is local with maximal ideal $\mathfrak{m}$. This was lead to a conjecture from Grothendieck that for any ideal $\mathfrak{a}$ of a Noetherian ring $R$ and any finite $R$-module $N$, the module $\operatorname{Hom}_R(R/\mathfrak{a}, \operatorname{H}^i_\mathfrak{a}(N))$ is finite \cite[Expose XIII, Conjecture 1.1]{G}. This conjecture is not true in general as shown by Hartshorne in \cite[Section 3]{Ha2}. However, he defined an $\mathfrak{a}$-torsion $R$-module $X$ to be \textit{$\mathfrak{a}$-cofinite} if $\operatorname{Ext}^{i}_{R}(R/\mathfrak{a}, X)$ is a finite $R$-module for all $i$ and asked, in \cite[First Question]{Ha2}, the following question:

\begin{ques}\label{1-1}
Is $\operatorname{H}^{i}_{\mathfrak{a}}(N)$ an $\mathfrak{a}$-cofinite $R$-module for all $i$?
\end{ques}

In \cite[Problem 4]{Hu}, Huneke raised the following question:

\begin{ques}\label{1-2}
Is $\operatorname{Ass}_R(\operatorname{H}^{i}_{\mathfrak{a}}(N))$ a finite set for all $i$?
\end{ques}

This question is related to Question \ref{1-1} because $\operatorname{Ass}_R(X)$ is a finite set for an $\mathfrak{a}$-cofinite $R$-module $X$. Singh, in \cite[Section 4]{Si}, has given a counterexample to these questions. However, Questions \ref{1-1} and \ref{1-2} have been studied by many authors and they were shown that these questions are true in some situations. In \cite[Theorem 7.10]{Mel1} and \cite[Theorem 2.10]{Mel}, Melkersson proved that $\operatorname{H}^{i}_{\mathfrak{a}}(N)$ is an $\mathfrak{a}$-cofinite $R$-module for all $i$ if $\dim(R)\leq 2$ or $\mathfrak{a}$ is an ideal of $R$ with $\dim(R/\mathfrak{a})\leq 1$. As an improvement of \cite[Theorem 7.10]{Mel1}, by \cite[Corollary 5.2]{CGH}, the answer to the above questions is yes whenever $\dim_R(N)\leq 2$.

Recall that a class of $R$-modules is a \textit{Serre subcategory} of the category of $R$-modules when it is closed under taking submodules, quotients, and extensions. Let $\mathcal{S}'$ and $\mathcal{S}''$ be two Serre subcategories of the category of $R$-modules. Yoshizawa, in \cite[Definition 1.1]{Yo}, defined \textit{the class of extension modules of $\mathcal{S}'$ by $\mathcal{S}''$} as the class of all $R$-modules $X$ with some $R$-modules $X'\in \mathcal{S}'$ and $X''\in \mathcal{S}'$ such that a sequence
\[0\longrightarrow X'\longrightarrow X\longrightarrow X''\longrightarrow 0\]
is exact. We will denote the class of extension modules of the class of all finite $R$-modules by the class of all $R$-modules $X$ with $\dim_R(X)< n$ by $\operatorname{FD}_{< n}$ (i.e., $X$ is an $\operatorname{FD}_{< n}$ $R$-module if there exists a finite $R$-submodule $X'$ of $X$ such that $\dim_R(X/X')< n$) (see \cite[Definition 2.1]{AB} and \cite[Definition 2.1]{AN}). From \cite[Theorem 2.3]{Yo}, $\operatorname{FD}_{< n}$ is a Serre subcategory of the category of $R$-modules. We say that $X$ is an \textit{$(\operatorname{FD}_{< n}, \mathfrak{a})$-cofinite $R$-module} if $X$ is an $\mathfrak{a}$-torsion $R$-module and $\operatorname{Ext}^{i}_{R}(R/\mathfrak{a}, X)$ is an $\operatorname{FD}_{< n}$ $R$-module for all $i$ \cite[Definition 4.1]{ATV}. Note that $\operatorname{FD}_{< 0}$ (resp. the class of all $(\operatorname{FD}_{< 0}, \mathfrak{a})$-cofinite $R$-modules) is the same as the class of all finite (resp. $\mathfrak{a}$-cofinite) $R$-modules. Therefore, as generalizations of Questions \ref{1-1} and \ref{1-2}, we have the following questions (see \cite[Question]{AbB}, \cite[Questions 1.6 and 1.8]{VM}, \cite[Questions 1.5 and 1.6]{VP}, and \cite[Questions 1.3 and 1.4]{VKhSh1}). In this paper, for a subset $A$ of $\operatorname{Spec}(R)$, the set $\{\mathfrak{p}\in A : \dim(R/\mathfrak{p})\geq n\}$ (resp. $\{\mathfrak{p}\in A : \dim(R/\mathfrak{p})= n\}$) is denoted by $A_{\geq n}$ (resp. $A_{= n}$).

\begin{ques}\label{1-3}
Is $\operatorname{H}^{i}_{\mathfrak{a}}(N)$ an $(\operatorname{FD}_{<  n}, \mathfrak{a})$-cofinite $R$-module for all $i$?
\end{ques}

\begin{ques}\label{1-4}
Is $\operatorname{Ass}_R(\operatorname{H}^{i}_{\mathfrak{a}}(N))_{\geq n}$ a finite set for all $i$?
\end{ques}

The first author and Morsali, in \cite[Corollary 4.5]{VM}, showed that $\operatorname{H}^{i}_{\mathfrak{a}}(N)$ is an $(\operatorname{FD}_{<  n}, \mathfrak{a})$-cofinite $R$-module and so $\operatorname{Ass}_R(\operatorname{H}^{i}_{\mathfrak{a}}(N))_{\geq n}$ is a finite set for all $i$ when $\dim(R/\mathfrak{a})\leq n+ 1$ which is a generalization of Melkersson's result \cite[Theorem 2.10]{Mel} (see also \cite[Theorems 2.5 and 2.10]{AbB} for the case that $R$ is a complete local ring). Also, the first author and Papari-Zarei, in \cite[Corollary 3.2]{VP}, proved that the answer to Questions \ref{1-3} and \ref{1-4} is yes if $\dim(R)\leq n+ 2$ which is a generalization of Melkersson's result \cite[Theorem 7.10]{Mel1}. As a generalization of \cite[Corollary 5.2]{CGH} and an improvement of \cite[Corollary 3.2]{VP}, the authors in \cite[Theorem 2.1 and Corollary 2.2]{VKhSh1} showed that the answer to Questions \ref{1-3} and \ref{1-4} is yes whenever $\dim_R(N)\leq n+ 2$ which also prepares an affirmative answer to Question \ref{1-2} for the case that $R$ is a semi-local ring and $\dim_R(N)\leq 3$ (see \cite[Corollary 2.4]{VKhSh1}).

Herzog introduced the \textit{$i$th generalized local cohomology module}
\[\operatorname{H}^i_\mathfrak{a}(Y, X)\cong \underset{j\in \mathbb{N}}\varinjlim \operatorname{Ext}^{i}_{R}(Y/{\mathfrak{a}}^{j}Y, X)\]
of $Y$ and $X$ with respect to $\mathfrak{a}$ in \cite{He}. Note that $\operatorname{H}^i_\mathfrak{a}(R, X)$ is just the ordinary local cohomology module $\operatorname{H}^i_\mathfrak{a}(X)$ of $X$ with respect to $\mathfrak{a}$. As generalizations of Questions \ref{1-1} and \ref{1-2}, we have the following questions for the theory of generalized local cohomology (see \cite[Question 2.7]{Ya}).

\begin{ques}\label{1-5}
Is $\operatorname{H}^{i}_{\mathfrak{a}}(M, N)$ an $\mathfrak{a}$-cofinite $R$-module for all $i$?
\end{ques}

\begin{ques}\label{1-6}
Is $\operatorname{Ass}_R(\operatorname{H}^{i}_{\mathfrak{a}}(M, N))$ a finite set for all $i$?
\end{ques}

The above questions have been studied by several authors and they were shown that these questions are true in some situations. Hassanzadeh and the first author, in \cite[Corollary 3.8]{HV}, generalized Melkersson's result \cite[Theorem 7.10]{Mel1} and showed that $\operatorname{H}^{i}_{\mathfrak{a}}(M, N)$ is an $\mathfrak{a}$-cofinite $R$-module for all $i$ if $\operatorname{pd}_R(M)< \infty$ ($\operatorname{pd}_R(M)$ is denoted as the projective dimension of $M$) and $\dim(R)\leq 2$. As an improvement of \cite[Corollary 3.8]{HV}, the answer to Questions \ref{1-5} and \ref{1-6} is yes whenever $\dim_R(M)\leq 2$ or $\dim_R(N)\leq 2$ by \cite[Theorem 1.3]{CGH}.

Now, as generalizations of Questions \ref{1-3} and \ref{1-4} and also Questions \ref{1-5} and \ref{1-6}, it is natural to raise the following questions (see \cite[Questions 1.10 and 1.12]{VP2} and \cite[Questions 1.10 and 1.12]{VM2}).

\begin{ques}\label{1-7}
Is $\operatorname{H}^{i}_{\mathfrak{a}}(M, N)$ an $(\operatorname{FD}_{<  n}, \mathfrak{a})$-cofinite $R$-module for all $i$?
\end{ques}

\begin{ques}\label{1-8}
Is $\operatorname{Ass}_R(\operatorname{H}^{i}_{\mathfrak{a}}(M, N))_{\geq n}$ a finite set for all $i$?
\end{ques}

The first author and Morsali, in \cite[Corollary 3.5]{VM2}, showed that $\operatorname{H}^{i}_{\mathfrak{a}}(M, N)$ is an $(\operatorname{FD}_{<  n}, \mathfrak{a})$-cofinite $R$-module and so $\operatorname{Ass}_R(\operatorname{H}^{i}_{\mathfrak{a}}(M, N))_{\geq n}$ is a finite set for all $i$ if $\dim(R/\mathfrak{a})\leq n+ 1$ which is a generalization of \cite[Corollary 4.5]{VM}. Also, the first author and Papari-Zarei, in \cite[Corollary 3.6]{VP2}, proved that the answer to Questions \ref{1-7} and \ref{1-8} is yes when $\dim(R)\leq n+ 2$ which improves and generalizes \cite[Corollary 3.2]{VP} and \cite[Corollary 3.8]{HV}.

In this paper, we continue studying cofiniteness and finiteness of associated prime ideals of generalized local cohomology modules and, in some situations, preparing affirmative answers to the above questions. As generalizations of \cite[Theorem 2.1 and Corollary 2.2]{VKhSh1} and \cite[Theorem 1.3]{CGH}, we show that $\operatorname{H}^{i}_{\mathfrak{a}}(M, N)$ is an $(\operatorname{FD}_{<  n}, \mathfrak{a})$-cofinite $R$-module and $\operatorname{Ass}_R(\operatorname{H}^{i}_{\mathfrak{a}}(M, N))_{\geq n}$ is a finite set for all $i$ if $\dim_R(M)\leq n+ 2$ or $\dim_R(N)\leq n+ 2$. As a consequence and as a generalization of \cite[Corollary 2.4]{VKhSh1}, $\operatorname{Ass}_R(\operatorname{H}^{i}_{\mathfrak{a}}(M, N))$ is a finite set for all $i$ when $R$ is a semi-local ring and $\dim_R(M)\leq 3$ or $\dim_R(N)\leq 3$. We prove that $\operatorname{H}^{i}_{\mathfrak{a}}(M, X)$ is an $(\operatorname{FD}_{< n}, \mathfrak{a})$-cofinite $R$-module, $\operatorname{Ass}_R(\operatorname{H}^{i}_{\mathfrak{a}}(M, X))_{\geq n}$ is a finite set, and $\operatorname{Ext}^{i}_{R}(M/\mathfrak{a}M, X)$ is an $\operatorname{FD}_{< n}$ $R$-module for all $i$ whenever $\operatorname{pd}_R(M)< \infty$, $\operatorname{H}^{i}_{\mathfrak{a}}(M, X)$ is an $\operatorname{FD}_{< n+ 2}$ $R$-module for all $i< \operatorname{pd}_R(M)+ \dim_R(X)- n$ $($resp. $i< \dim(R)- n)$, and $\operatorname{Ext}^{i}_{R}(M/\mathfrak{a}M, X)$ is an $\operatorname{FD}_{< n}$ $R$-module for all $i\leq \operatorname{pd}_R(M)+ \dim_R(X)- n$ $($resp. $i\leq \dim(R)- n)$. We show that if $\dim(R/\mathfrak{a})\leq 2$, $\operatorname{Ext}^{i}_{R}(M/\mathfrak{a}M, X)$ is a finite $R$-module for all $i\leq t+ 1$, $\operatorname{Supp}_R(M)\cap \operatorname{Supp}_R(X)\cap \operatorname{Var}(\mathfrak{a})\cap\operatorname{Max}(R)$ is a finite set, and $\operatorname{Hom}_{R}(R/\mathfrak{a}, \operatorname{H}^j_{\mathfrak{a}}(M, X))$ is a finite $R$-module for all $j\leq t$, then $\operatorname{H}^j_{\mathfrak{a}}(M, X)$ is an $\mathfrak{a}$-cofinite $R$-module for all $j< t$. We also prove that $\operatorname{H}^j_{\mathfrak{a}}(M, X)$ is an $\mathfrak{a}$-cofinite $R$-module for all $j$ when $\dim(R/\mathfrak{a})\leq 2$, $\operatorname{Ext}^{i}_{R}(M/\mathfrak{a}M, X)$ is a finite $R$-module for all $i$, $\operatorname{Supp}_R(M)\cap \operatorname{Supp}_R(X)\cap \operatorname{Var}(\mathfrak{a})\cap\operatorname{Max}(R)$ is a finite set, and $\operatorname{H}^{2j}_{\mathfrak{a}}(M, X)$ (resp. $\operatorname{H}^{2j+ 1}_{\mathfrak{a}}(M, X)$) is an $\mathfrak{a}$-cofinite $R$-module for all $j$.
%%% ----------------------------------------------------------------------
%%% ----------------------------------------------------------------------
%%% ----------------------------------------------------------------------
%%%
%%%
%%%
\section{Main results}\label{2}
%%%
%%%
%%%
%%% ----------------------------------------------------------------------
%%% ----------------------------------------------------------------------
%%% ----------------------------------------------------------------------

%%% ----------------------------------------------------------------------

%%% ----------------------------------------------------------------------

%   {2.1}   ----------------------------------------------------------------------
In the first result, we prepare affirmative answers to Questions \ref{1-7} and \ref{1-8} for the case that $\dim_R(M)\leq n+ 2$ or $\dim_R(N)\leq n+ 2$ which is generalizations of \cite[Theorem 2.1 and Corollary 2.2]{VKhSh1} and \cite[Theorem 1.3]{CGH} (see also \cite[Theorem 7.10]{Mel1}, \cite[Corollary 3.2]{VP}, \cite[Corollary 3.8]{HV}, and \cite[Corollary 3.6]{VP2}).

\begin{thm}\label{2-1}
Let $M$ and $N$ be two finite $R$-modules such that $\dim_R(M)\leq n+ 2$ or $\dim_R(N)\leq n+ 2$. Then the following statements hold true:
\begin{itemize}
\item[\emph{(i)}] $\operatorname{H}^{i}_{\mathfrak{a}}(M, N)$ is an $(\operatorname{FD}_{< n}, \mathfrak{a})$-cofinite $\operatorname{FD}_{< n+ 2}$ $R$-module for all $i$;
\item[\emph{(ii)}] $\operatorname{Ass}_R(\operatorname{H}^{i}_{\mathfrak{a}}(M, N))_{\geq n}$ is a finite set for all $i$.
\end{itemize}
\end{thm}

\begin{proof}
(i). We first assume that $\dim_R(M)\leq n+ 2$. Set $\overline{M}= M/\Gamma_\mathfrak{a}(M)$. Since $\Gamma_\mathfrak{a}(\overline{M})= 0$, the ideal $\mathfrak{a}$ contains an element $a$ which is a non-zerodivisor on $\overline{M}$. Let $t$ be a non-negative integer. Since $a$ is not in any minimal member of $\operatorname{Supp}_R(\overline{M})$, we have $\dim_R(\overline{M}/a\overline{M})\leq n+ 1$ and so $\dim_R(\operatorname{H}^{t}_{\mathfrak{a}}(\overline{M}/a\overline{M}, N))\leq n+ 1$. The short exact sequence
\[0\longrightarrow \overline{M}\overset{a}\longrightarrow \overline{M}\longrightarrow \overline{M}/a\overline{M}\longrightarrow 0\]
induces an exact sequence
\[\operatorname{H}^{t}_{\mathfrak{a}}(\overline{M}/a\overline{M}, N)\longrightarrow \operatorname{H}^{t}_{\mathfrak{a}}(\overline{M}, N)\overset{a}\longrightarrow \operatorname{H}^{t}_{\mathfrak{a}}(\overline{M}, N).\]
Thus $\dim_R(\operatorname{Hom}_R(R/Ra, \operatorname{H}^{t}_{\mathfrak{a}}(\overline{M}, N)))\leq n+ 1$ and so $\dim_R(\operatorname{H}^{t}_{\mathfrak{a}}(\overline{M}, N))\leq n+ 1$ because, by \cite[Exercise 1.2.28]{BH}, $\operatorname{Ass}_R(\operatorname{Hom}_R(R/Ra, \operatorname{H}^{t}_{\mathfrak{a}}(\overline{M}, N)))= \operatorname{Ass}_R(\operatorname{H}^{t}_{\mathfrak{a}}(\overline{M}, N))$. Hence $\operatorname{H}^{t}_{\mathfrak{a}}(\overline{M}, N)$ is an $\operatorname{FD}_{< n+ 2}$ $R$-module. From the short exact sequence
\[0\longrightarrow \Gamma_\mathfrak{a}(M)\longrightarrow M\longrightarrow \overline{M}\longrightarrow 0\]
and \cite[Lemma 2.5(c)]{VA}, we have the exact sequence
\[\operatorname{H}^{t}_{\mathfrak{a}}(\overline{M}, N)\longrightarrow \operatorname{H}^{t}_{\mathfrak{a}}(M, N)\longrightarrow \operatorname{Ext}^{t}_{R}(\Gamma_\mathfrak{a}(M), N)\]
which shows that $\operatorname{H}^{t}_{\mathfrak{a}}(M, N)$ is an $\operatorname{FD}_{< n+ 2}$ $R$-module. Therefore $\operatorname{H}^{i}_{\mathfrak{a}}(M, N)$ is an $(\operatorname{FD}_{< n}, \mathfrak{a})$-cofinite $R$-module for all $i$ from \cite[Theorem 3.2]{VM2}.

Now, we assume that $\dim_R(N)\leq n+ 2$. Set $\overline{N}= N/\Gamma_\mathfrak{a}(N)$ and let $t$ be a non-negative integer. Then there is an element $a$ of $\mathfrak{a}$ which is a non-zerodivisor on $\overline{N}$ and so $\dim_R(\overline{N}/a\overline{N})\leq n+ 1$. Therefore $\dim_R(\operatorname{H}^{t- 1}_{\mathfrak{a}}(M, \overline{N}/a\overline{N}))\leq n+ 1$. From the short exact sequence
\[0\longrightarrow \overline{N}\overset{a}\longrightarrow \overline{N}\longrightarrow \overline{N}/a\overline{N}\longrightarrow 0\]
we get the exact sequence
\[\operatorname{H}^{t- 1}_{\mathfrak{a}}(M, \overline{N}/a\overline{N})\longrightarrow \operatorname{H}^{t}_{\mathfrak{a}}(M, \overline{N})\overset{a}\longrightarrow \operatorname{H}^{t}_{\mathfrak{a}}(M, \overline{N}).\]
Thus $\dim_R(\operatorname{Hom}_R(R/Ra, \operatorname{H}^{t}_{\mathfrak{a}}(M, \overline{N})))\leq n+ 1$ and so $\dim_R(\operatorname{H}^{t}_{\mathfrak{a}}(M, \overline{N}))\leq n+ 1$. Hence $\operatorname{H}^{t}_{\mathfrak{a}}(M, \overline{N})$ is an $\operatorname{FD}_{< n+ 2}$ $R$-module. By the short exact sequence
\[0\longrightarrow \Gamma_\mathfrak{a}(N)\longrightarrow N\longrightarrow \overline{N}\longrightarrow 0\]
and \cite[Lemma 2.5(c)]{VA}, we have the exact sequence
\[\operatorname{Ext}^{t}_{R}(M, \Gamma_\mathfrak{a}(N))\longrightarrow \operatorname{H}^{t}_{\mathfrak{a}}(M, N)\longrightarrow \operatorname{H}^{t}_{\mathfrak{a}}(M, \overline{N}).\]
It shows that $\operatorname{H}^{t}_{\mathfrak{a}}(M, N)$ is an $\operatorname{FD}_{< n+ 2}$ $R$-module. Therefore $\operatorname{H}^{i}_{\mathfrak{a}}(M, N)$ is an $(\operatorname{FD}_{< n}, \mathfrak{a})$-cofinite $R$-module for all $i$ from \cite[Theorem 3.2]{VM2}.

(ii). By the first part, $\operatorname{Hom}_R(R/\mathfrak{a},\operatorname{H}^{i}_{\mathfrak{a}}(M, N))$ is an $\operatorname{FD}_{< n}$ $R$-module for all $i$ and hence $\operatorname{Ass}_R(\operatorname{Hom}_R(R/\mathfrak{a},\operatorname{H}^{i}_{\mathfrak{a}}(M, N)))_{\geq n}$ is finite for all $i$. Thus $\operatorname{Ass}_R(\operatorname{H}^{i}_{\mathfrak{a}}(M, N))_{\geq n}$ is a finite set for all $i$ from \cite[Exercise 1.2.28]{BH}.
\end{proof}
%%% ----------------------------------------------------------------------

%%% ----------------------------------------------------------------------

%   {2.2}   ----------------------------------------------------------------------
\begin{cor}\label{2-2}
Let $M$ and $N$ be two finite $R$-modules such that $\dim_R(M)\leq 2$ or $\dim_R(N)\leq 2$. Then $\operatorname{H}^{i}_{\mathfrak{a}}(M, N)$ is an $\mathfrak{a}$-cofinite $\operatorname{FD}_{< 2}$ $R$-module and $\operatorname{Ass}_R(\operatorname{H}^{i}_{\mathfrak{a}}(M, N))$ is a finite set for all $i$.
\end{cor}

\begin{proof}
Take $n= 0$ in Theorem \ref{2-1}.
\end{proof}
%%% ----------------------------------------------------------------------

%%% ----------------------------------------------------------------------

%   {2.3}   ----------------------------------------------------------------------
Recall that $X$ is said to be a \textit{weakly Laskerian $R$-module} if the set of associated prime ideals of any quotient module of $X$ is finite \cite[Definition 2.1]{DM1}. Also, we say that $X$ is an \textit{$\mathfrak{a}$-weakly cofinite $R$-module} if $X$ is an $\mathfrak{a}$-torsion $R$-module and $\operatorname{Ext}^{i}_{R}(R/\mathfrak{a}, X)$ is a weakly Laskerian $R$-module for all $i$ \cite[Definition 2.4]{DM2}. 
The next corollary prepares affirmative answer to Question \ref{1-6} for the case that $R$ is semi-local and $\dim_R(M)\leq 3$ or $\dim_R(N)\leq 3$ (see \cite[Corollary 5.3]{CGH} for the case that $R$ is a local ring).

\begin{cor}\label{2-3}
Let $R$ be a semi-local ring and let $M$ and $N$ be two finite $R$-modules such that $\dim_R(M)\leq 3$ or $\dim_R(N)\leq 3$. Then $\operatorname{H}^{i}_{\mathfrak{a}}(M, N)$ is an $\mathfrak{a}$-weakly cofinite $\operatorname{FD}_{< 3}$ $R$-module and $\operatorname{Ass}_R(\operatorname{H}^{i}_{\mathfrak{a}}(M, N))$ is a finite set for all $i$.
\end{cor}

\begin{proof}
By considering \cite[Theorem 3.3]{B}, put $n= 1$ in Theorem \ref{2-1}.
\end{proof}
%%% ----------------------------------------------------------------------

%%% ----------------------------------------------------------------------

%   {2.4}   ----------------------------------------------------------------------
By taking $M=R$ in Theorem \ref{2-1} and Corollaries \ref{2-2} and \ref{2-3}, we have the following results for the ordinary local cohomology modules.

\begin{cor}\label{2-4}
Let $N$ be a finite $R$-module such that $\dim_R(N)\leq n+ 2$. Then the following statements hold true:
\begin{itemize}
\item[\emph{(i)}] $\operatorname{H}^{i}_{\mathfrak{a}}(N)$ is an $(\operatorname{FD}_{< n}, \mathfrak{a})$-cofinite $\operatorname{FD}_{< n+ 2}$ $R$-module for all $i$;
\item[\emph{(ii)}] $\operatorname{Ass}_R(\operatorname{H}^{i}_{\mathfrak{a}}(N))_{\geq n}$ is a finite set for all $i$.
\end{itemize}
\end{cor}
%%% ----------------------------------------------------------------------

%%% ----------------------------------------------------------------------

%   {2.5}   ----------------------------------------------------------------------
\begin{cor}\label{2-5}
Let $N$ be a finite $R$-module such that $\dim_R(N)\leq 2$. Then $\operatorname{H}^{i}_{\mathfrak{a}}(N)$ is an $\mathfrak{a}$-cofinite $\operatorname{FD}_{< 2}$ $R$-module and $\operatorname{Ass}_R(\operatorname{H}^{i}_{\mathfrak{a}}(N))$ is a finite set for all $i$.
\end{cor}
%%% ----------------------------------------------------------------------

%%% ----------------------------------------------------------------------

%   {2.6}   ----------------------------------------------------------------------
\begin{cor}\label{2-6}
Let $R$ be a semi-local ring and let $N$ be a finite $R$-module such that $\dim_R(N)\leq 3$. Then $\operatorname{H}^{i}_{\mathfrak{a}}(N)$ is an $\mathfrak{a}$-weakly cofinite $\operatorname{FD}_{< 3}$ $R$-module and $\operatorname{Ass}_R(\operatorname{H}^{i}_{\mathfrak{a}}(N))$ is a finite set for all $i$.
\end{cor}
%%% ----------------------------------------------------------------------

%%% ----------------------------------------------------------------------

%   {2.7}   ----------------------------------------------------------------------
From \cite[Corollary 2.6]{BNS}, $\operatorname{Ext}^{i}_{R}(R/\mathfrak{a}, X)$ is a finite $R$-module for all $i$ whenever $\dim(R/\mathfrak{a})= 1$ and $\operatorname{Ext}^{i}_{R}(R/\mathfrak{a}, X)$ is a finite $R$-module for all $i\leq \dim_R(X)$. In \cite[Theorem 2.5]{VKhSh1}, we generalized and improved \cite[Corollary 2.6]{BNS} by showing that $\operatorname{Ext}^{i}_{R}(R/\mathfrak{a}, X)$ is an $\operatorname{FD}_{< n}$ $R$-module for all $i$ if $\operatorname{H}^{i}_{\mathfrak{a}}(X)$ is an $\operatorname{FD}_{< n+ 2}$ $R$-module for all $i< \dim_R(X)- n$ $($e.g., $\dim(R/\mathfrak{a})\leq n+ 1)$ and $\operatorname{Ext}^{i}_{R}(R/\mathfrak{a}, X)$ is an $\operatorname{FD}_{< n}$ $R$-module for all $i\leq \dim_R(X)- n$. As a generalization of \cite[Theorem 2.5]{VKhSh1}, we have the following theorem which is also related to Questions \ref{1-7} and \ref{1-8}.

\begin{thm}\label{2-7}
Let $M$ be a finite $R$-module with $\operatorname{pd}_R(M)< \infty$ and let $X$ be an arbitrary $R$-module such that $\operatorname{H}^{i}_{\mathfrak{a}}(M, X)$ is an $\operatorname{FD}_{< n+ 2}$ $R$-module for all $i< \operatorname{pd}_R(M)+ \dim_R(X)- n$ $($resp. $i< \dim(R)- n)$ $($e.g., $\dim(R/\mathfrak{a})\leq n+ 1)$ and $\operatorname{Ext}^{i}_{R}(M/\mathfrak{a}M, X)$ is an $\operatorname{FD}_{< n}$ $R$-module for all $i\leq \operatorname{pd}_R(M)+ \dim_R(X)- n$ $($resp. $i\leq \dim(R)- n)$. Then the following statements hold true:
\begin{itemize}
\item[\emph{(i)}] $\operatorname{H}^{i}_{\mathfrak{a}}(M, X)$ is an $(\operatorname{FD}_{< n}, \mathfrak{a})$-cofinite $R$-module for all $i$;
\item[\emph{(ii)}] $\operatorname{Ass}_R(\operatorname{H}^{i}_{\mathfrak{a}}(M, X))_{\geq n}$ is a finite set for all $i$;
\item[\emph{(iii)}] $\operatorname{Ext}^{i}_{R}(M/\mathfrak{a}M, X)$ is an $\operatorname{FD}_{< n}$ $R$-module for all $i$.
\end{itemize}
\end{thm}

\begin{proof}
We first show that for a non-negative integer $t$, $\operatorname{H}^{t}_{\mathfrak{a}}(M, X)$ is an $\operatorname{FD}_{< \operatorname{pd}_R(M)+ \dim_R(X)- t+ 1}$ $($resp. $\operatorname{FD}_{< \dim(R)- t+ 1})$ $R$-module. Assume contrarily that the $R$-module $\operatorname{H}^{t}_{\mathfrak{a}}(M, X)$ is not $\operatorname{FD}_{< \operatorname{pd}_R(M)+ \dim_R(X)- t+ 1}$ (resp. $\operatorname{FD}_{< \dim(R)- t+ 1}$). Then $\dim_R(\operatorname{H}^{t}_{\mathfrak{a}}(M, X))> \operatorname{pd}_R(M)+ \dim_R(X)- t$ (resp. $\dim_R(\operatorname{H}^{t}_{\mathfrak{a}}(M, X))> \dim(R)- t$) and so there exists a prime ideal $\mathfrak{p}$ of $\operatorname{Supp}_R(\operatorname{H}^{t}_{\mathfrak{a}}(M, X))$ such that $\dim(R/\mathfrak{p})> \operatorname{pd}_R(M)+ \dim_R(X)- t$ (resp. $\dim(R/\mathfrak{p})> \dim(R)- t$). Hence, by \cite[Lemma 2.5(b)]{VA}, $\operatorname{H}^{t}_{\mathfrak{a}R_\mathfrak{p}}(M_\mathfrak{p}, X_\mathfrak{p})\neq 0$ and $t> \operatorname{pd}_{R_\mathfrak{p}}(M_\mathfrak{p})+ \dim_{R_\mathfrak{p}}(X_\mathfrak{p})$ (resp. $t> \dim(R_\mathfrak{p})$) which contradicts \cite[Proposition 2.8]{HV} (resp. \cite[Corollary 3.2]{CH}).

(i). From the first paragraph of the proof, $\operatorname{H}^{i}_{\mathfrak{a}}(M, X)$ is an $\operatorname{FD}_{< n}$ $R$-module and so is an $(\operatorname{FD}_{< n}, \mathfrak{a})$-cofinite $R$-module for all $i> \operatorname{pd}_R(M)+ \dim_R(X)- n$ (resp. $i> \dim(R)- n$). Also, by \cite[Theorem 3.2(i)]{VM2}, $\operatorname{H}^{i}_{\mathfrak{a}}(M, X)$ is an $(\operatorname{FD}_{< n}, \mathfrak{a})$-cofinite $R$-module for all $i< \operatorname{pd}_R(M)+ \dim_R(X)- n$ (resp. $i< \dim(R)- n$). Therefore $\operatorname{Hom}_{R}(R/\mathfrak{a}, \operatorname{H}^{\operatorname{pd}_R(M)+ \dim_R(X)- n}_{\mathfrak{a}}(M, X))$ (resp. $\operatorname{Hom}_{R}(R/\mathfrak{a}, \operatorname{H}^{\dim(R)- n}_{\mathfrak{a}}(M, X))$) is an $\operatorname{FD}_{< n}$ $R$-module from \cite[Lemma 2.4]{VP2}. On the other hand, again by the first paragraph of the proof, $\operatorname{H}^{\operatorname{pd}_R(M)+ \dim_R(X)- n}_{\mathfrak{a}}(M, X)$ (resp. $\operatorname{H}^{\dim(R)- n}_{\mathfrak{a}}(M, X)$) is an $\operatorname{FD}_{< n+ 1}$ $R$-module. Thus $\operatorname{H}^{\operatorname{pd}_R(M)+ \dim_R(X)- n}_{\mathfrak{a}}(M, X)$ (resp. $\operatorname{H}^{\dim(R)- n}_{\mathfrak{a}}(M, X)$) is an $(\operatorname{FD}_{< n}, \mathfrak{a})$-cofinite $R$-module from \cite[Lemma 2.1]{VM}.

(ii). This follows by the first part and \cite[Exercise 1.2.28]{BH}.

(iii). It follows from the first part, \cite[Lemma 2.3]{VP2}, and using an induction argument on $i$.
\end{proof}
%%% ----------------------------------------------------------------------

%%% ----------------------------------------------------------------------

%   {2.8}   ----------------------------------------------------------------------
As immediate applications of the above theorem, we have the following corollaries.

\begin{cor}\label{2-8}
Let $M$ be a finite $R$-module with $\operatorname{pd}_R(M)< \infty$ and let $X$ be an arbitrary $R$-module such that $\operatorname{H}^{i}_{\mathfrak{a}}(M, X)$ is an $\operatorname{FD}_{< 2}$ $R$-module for all $i< \operatorname{pd}_R(M)+ \dim_R(X)$ $($resp. $i< \dim(R))$ $($e.g., $\dim(R/\mathfrak{a})\leq 1)$ and $\operatorname{Ext}^{i}_{R}(M/\mathfrak{a}M, X)$ is a finite $R$-module for all $i\leq \operatorname{pd}_R(M)+ \dim_R(X)$ $($resp. $i\leq \dim(R))$. Then, for all $i$, $\operatorname{H}^{i}_{\mathfrak{a}}(M, X)$ is an $\mathfrak{a}$-cofinite $R$-module, $\operatorname{Ass}_R(\operatorname{H}^{i}_{\mathfrak{a}}(M, X))$ is a finite set, and $\operatorname{Ext}^{i}_{R}(M/\mathfrak{a}M, X)$ is a finite $R$-module.
\end{cor}

\begin{proof}
Take $n= 0$ in Theorem \ref{2-7}.
\end{proof}
%%% ----------------------------------------------------------------------

%%% ----------------------------------------------------------------------

%   {2.9}   ----------------------------------------------------------------------
\begin{cor}\label{2-9}
Let $R$ be a semi-local ring, let $M$ be a finite $R$-module with $\operatorname{pd}_R(M)< \infty$, and let $X$ be an arbitrary $R$-module such that $\operatorname{H}^{i}_{\mathfrak{a}}(M, X)$ is an $\operatorname{FD}_{< 3}$ $R$-module for all $i< \operatorname{pd}_R(M)+ \dim_R(X)- 1$ $($resp. $i< \dim(R)- 1)$ $($e.g., $\dim(R/\mathfrak{a})\leq 2)$ and $\operatorname{Ext}^{i}_{R}(M/\mathfrak{a}M, X)$ is an $\operatorname{FD}_{< 1}$ $R$-module for all $i\leq \operatorname{pd}_R(M)+ \dim_R(X)- 1$ $($resp. $i\leq \dim(R)- 1)$. Then, for all $i$, $\operatorname{H}^{i}_{\mathfrak{a}}(M, X)$ is an $\mathfrak{a}$-weakly cofinite $R$-module, $\operatorname{Ass}_R(\operatorname{H}^{i}_{\mathfrak{a}}(M, X))$ is a finite set, and $\operatorname{Ext}^{i}_{R}(M/\mathfrak{a}M, X)$ is a weakly Laskerian $R$-module.
\end{cor}

\begin{proof}
Consider \cite[Theorem 3.3]{B} and put $n= 1$ in Theorem \ref{2-7}.
\end{proof}
%%% ----------------------------------------------------------------------

%%% ----------------------------------------------------------------------

%   {2.10}   ----------------------------------------------------------------------
\begin{cor}\label{2-10}
Let $M$ be a finite $R$-module with $\operatorname{pd}_R(M)< \infty$ and let $X$ be an arbitrary $R$-module such that $\operatorname{Ext}^{i}_{R}(M/\mathfrak{a}M, X)$ is an $\operatorname{FD}_{< n}$ $R$-module for all $i\leq \operatorname{pd}_R(M)+ \dim_R(X)- n$ $($resp. $i\leq \dim(R)- n)$. Then, for every integer $i$ and every prime ideal $\mathfrak{p}$ of $\operatorname{Var}(\mathfrak{a})_{\geq n}$, $\operatorname{H}^{i}_{\mathfrak{p}R_\mathfrak{p}}(M_\mathfrak{p}, X_\mathfrak{p})$ is a $\mathfrak{p}R_\mathfrak{p}$-cofinite $R_\mathfrak{p}$-module, $\operatorname{Ass}_{R_\mathfrak{p}}(\operatorname{H}^{i}_{\mathfrak{p}R_\mathfrak{p}}(M_\mathfrak{p}, X_\mathfrak{p}))$ is a finite set, and $\operatorname{Ext}^{i}_{R_\mathfrak{p}}(M_\mathfrak{p}/\mathfrak{p}M_\mathfrak{p}, X_\mathfrak{p})$ is a finite $R_\mathfrak{p}$-module.
\end{cor}

\begin{proof}
Assume that $\mathfrak{p}$ is a prime ideal of $\operatorname{Var}(\mathfrak{a})_{\geq n}$. Then, from \cite[Lemma 2.2]{VP2}, $\operatorname{Ext}^{i}_{R}(M/\mathfrak{p}M, X)$ is an $\operatorname{FD}_{< n}$ $R$-module for all $i\leq \operatorname{pd}_R(M)+ \dim_R(X)- n$ $($resp. $i\leq \dim(R)- n)$. Thus $\operatorname{Ext}^{i}_{R_\mathfrak{p}}(M_\mathfrak{p}/\mathfrak{p}M_\mathfrak{p}, X_\mathfrak{p})$ is a finite $R_\mathfrak{p}$-module for all $i\leq \operatorname{pd}_{R_\mathfrak{p}}(M_\mathfrak{p})+ \dim_{R_\mathfrak{p}}(X_\mathfrak{p})$ $($resp. $i\leq \dim(R_\mathfrak{p}))$. Hence, for all $i$, $\operatorname{H}^{i}_{\mathfrak{p}R_\mathfrak{p}}(M_\mathfrak{p}, X_\mathfrak{p})$ is a $\mathfrak{p}R_\mathfrak{p}$-cofinite $R_\mathfrak{p}$-module, $\operatorname{Ass}_{R_\mathfrak{p}}(\operatorname{H}^{i}_{\mathfrak{p}R_\mathfrak{p}}(M_\mathfrak{p}, X_\mathfrak{p}))$ is a finite set, and $\operatorname{Ext}^{i}_{R_\mathfrak{p}}(M_\mathfrak{p}/\mathfrak{p}M_\mathfrak{p}, X_\mathfrak{p})$ is a finite $R_\mathfrak{p}$-module by Corollary \ref{2-8}.
\end{proof}
%%% ----------------------------------------------------------------------

%%% ----------------------------------------------------------------------

%   {2.11}   ----------------------------------------------------------------------
\begin{cor}\label{2-11}
Let $M$ be a finite $R$-module with $\operatorname{pd}_R(M)< \infty$ and let $X$ be an arbitrary $R$-module such that $\operatorname{Ext}^{i}_{R}(M/\mathfrak{a}M, X)$ is a finite $R$-module for all $i\leq \operatorname{pd}_R(M)+ \dim_R(X)$ $($resp. $i\leq \dim(R))$. Then, for every integer $i$ and every prime ideal $\mathfrak{p}$ of $\operatorname{Var}(\mathfrak{a})$, $\operatorname{H}^{i}_{\mathfrak{p}R_\mathfrak{p}}(M_\mathfrak{p}, X_\mathfrak{p})$ is a $\mathfrak{p}R_\mathfrak{p}$-cofinite $R_\mathfrak{p}$-module, $\operatorname{Ass}_{R_\mathfrak{p}}(\operatorname{H}^{i}_{\mathfrak{p}R_\mathfrak{p}}(M_\mathfrak{p}, X_\mathfrak{p}))$ is a finite set, and $\operatorname{Ext}^{i}_{R_\mathfrak{p}}(M_\mathfrak{p}/\mathfrak{p}M_\mathfrak{p}, X_\mathfrak{p})$ is a finite $R_\mathfrak{p}$-module.
\end{cor}

\begin{proof}
Take $n= 0$ in Corollary \ref{2-10}.
\end{proof}
%%% ----------------------------------------------------------------------

%%% ----------------------------------------------------------------------

%   {2.12}   ----------------------------------------------------------------------
We have the following corollaries for the ordinary local cohomology modules by putting $M=R$ in Theorem \ref{2-7} and Corollaries \ref{2-8}--\ref{2-11}.

\begin{cor}\label{2-12}
Let $X$ be an arbitrary $R$-module such that $\operatorname{H}^{i}_{\mathfrak{a}}(X)$ is an $\operatorname{FD}_{< n+ 2}$ $R$-module for all $i< \dim_R(X)- n$ $($e.g., $\dim(R/\mathfrak{a})\leq n+ 1)$ and $\operatorname{Ext}^{i}_{R}(R/\mathfrak{a}, X)$ is an $\operatorname{FD}_{< n}$ $R$-module for all $i\leq \dim_R(X)- n$. Then the following statements hold true:
\begin{itemize}
\item[\emph{(i)}] $\operatorname{H}^{i}_{\mathfrak{a}}(X)$ is an $(\operatorname{FD}_{< n}, \mathfrak{a})$-cofinite $R$-module for all $i$;
\item[\emph{(ii)}] $\operatorname{Ass}_R(\operatorname{H}^{i}_{\mathfrak{a}}(X))_{\geq n}$ is a finite set for all $i$;
\item[\emph{(iii)}] $\operatorname{Ext}^{i}_{R}(R/\mathfrak{a}, X)$ is an $\operatorname{FD}_{< n}$ $R$-module for all $i$.
\end{itemize}
\end{cor}
%%% ----------------------------------------------------------------------

%%% ----------------------------------------------------------------------

%   {2.13}   ----------------------------------------------------------------------
\begin{cor}\label{2-13}
Let $X$ be an arbitrary $R$-module such that $\operatorname{H}^{i}_{\mathfrak{a}}(X)$ is an $\operatorname{FD}_{< 2}$ $R$-module for all $i< \dim_R(X)$ $($e.g., $\dim(R/\mathfrak{a})\leq 1)$ and $\operatorname{Ext}^{i}_{R}(R/\mathfrak{a}, X)$ is a finite $R$-module for all $i\leq \dim_R(X)$. Then, for all $i$, $\operatorname{H}^{i}_{\mathfrak{a}}(X)$ is an $\mathfrak{a}$-cofinite $R$-module, $\operatorname{Ass}_R(\operatorname{H}^{i}_{\mathfrak{a}}(X))$ is a finite set, and $\operatorname{Ext}^{i}_{R}(R/\mathfrak{a}, X)$ is a finite $R$-module.
\end{cor}
%%% ----------------------------------------------------------------------

%%% ----------------------------------------------------------------------

%   {2.14}   ----------------------------------------------------------------------
\begin{cor}\label{2-14}
Let $R$ be a semi-local ring and let $X$ be an arbitrary $R$-module such that $\operatorname{H}^{i}_{\mathfrak{a}}(X)$ is an $\operatorname{FD}_{< 3}$ $R$-module for all $i< \dim_R(X)- 1$ $($e.g., $\dim(R/\mathfrak{a})\leq 2)$ and $\operatorname{Ext}^{i}_{R}(R/\mathfrak{a}, X)$ is an $\operatorname{FD}_{< 1}$ $R$-module for all $i\leq \dim_R(X)- 1$. Then, for all $i$, $\operatorname{H}^{i}_{\mathfrak{a}}(X)$ is an $\mathfrak{a}$-weakly cofinite $R$-module, $\operatorname{Ass}_R(\operatorname{H}^{i}_{\mathfrak{a}}(X))$ is a finite set, and $\operatorname{Ext}^{i}_{R}(R/\mathfrak{a}, X)$ is a weakly Laskerian $R$-module.
\end{cor}
%%% ----------------------------------------------------------------------

%%% ----------------------------------------------------------------------

%   {2.15}   ----------------------------------------------------------------------
For the last parts of the next two corollaries (see \cite[Corollaries 2.9  and 2.10]{VKhSh1}), we use \cite[Theorem 2.1]{Mel1} which shows that, for every prime ideal $\mathfrak{p}$ of $\operatorname{Spec}(R)$, the $i$th Bass number of $X$ with respect to $\mathfrak{p}$ is finite for all $i$ if and only if the $i$th Betti number of $X$ with respect to $\mathfrak{p}$ is finite for all $i$.

\begin{cor}\label{2-15}
Let $X$ be an arbitrary $R$-module such that $\operatorname{Ext}^{i}_{R}(R/\mathfrak{a}, X)$ is an $\operatorname{FD}_{< n}$ $R$-module for all $i\leq \dim_R(X)- n$. Then, for every integer $i$ and every prime ideal $\mathfrak{p}$ of $\operatorname{Var}(\mathfrak{a})_{\geq n}$, $\operatorname{H}^{i}_{\mathfrak{p}R_\mathfrak{p}}(X_\mathfrak{p})$ is a $\mathfrak{p}R_\mathfrak{p}$-cofinite $R_\mathfrak{p}$-module, $\operatorname{Ass}_{R_\mathfrak{p}}(\operatorname{H}^{i}_{\mathfrak{p}R_\mathfrak{p}}(X_\mathfrak{p}))$ is a finite set, and the $i$th Bass number and the $i$th Betti number of $X$ with respect to $\mathfrak{p}$ are finite.
\end{cor}
%%% ----------------------------------------------------------------------

%%% ----------------------------------------------------------------------

%   {2.16}   ----------------------------------------------------------------------
\begin{cor}\label{2-16}
Let $X$ be an arbitrary $R$-module such that $\operatorname{Ext}^{i}_{R}(R/\mathfrak{a}, X)$ is a finite $R$-module for all $i\leq \dim_R(X)$. Then, for every integer $i$ and every prime ideal $\mathfrak{p}$ of $\operatorname{Var}(\mathfrak{a})$, $\operatorname{H}^{i}_{\mathfrak{p}R_\mathfrak{p}}(X_\mathfrak{p})$ is a $\mathfrak{p}R_\mathfrak{p}$-cofinite $R_\mathfrak{p}$-module, $\operatorname{Ass}_{R_\mathfrak{p}}(\operatorname{H}^{i}_{\mathfrak{p}R_\mathfrak{p}}(X_\mathfrak{p}))$ is a finite set, and the $i$th Bass number and the $i$th Betti number of $X$ with respect to $\mathfrak{p}$ are finite.
\end{cor}
%%% ----------------------------------------------------------------------

%%% ----------------------------------------------------------------------

%   {2.17}   ----------------------------------------------------------------------
Assume that $\mathfrak{a}$ is an ideal of $R$ with $\dim(R/\mathfrak{a})\leq 2$. If $R$ is a local ring and $N$ is a finite $R$-module, in \cite[Theorem 3.7]{BNS} (resp. \cite[Theorem 3.8]{BNS}), it is shown that $\operatorname{H}^j_{\mathfrak{a}}(N)$ is an $\mathfrak{a}$-cofinite $R$-module for all $j< t$ (resp. for all $j$) whenever $\operatorname{Hom}_{R}(R/\mathfrak{a}, \operatorname{H}^j_{\mathfrak{a}}(N))$ (resp. $\operatorname{H}^{2j}_{\mathfrak{a}}(N)$ or $\operatorname{H}^{2j+ 1}_{\mathfrak{a}}(N)$) is a finite (resp. an $\mathfrak{a}$-cofinite) $R$-module for all $j\leq t$ (resp. for all $j$). In \cite[Corollaries 2.12 and 2.13]{VKhSh1}, we improved \cite[Theorems 3.7 and 3.8]{BNS} to the rings which are not necessarily local and to the modules which are not necessarily finite. As generalizations of \cite[Corollaries 2.12 and 2.13]{VKhSh1} to the generalized local cohomology modules, we have the following theorems which are related to Question \ref{1-5}.

\begin{thm}\label{2-17}
Let $\mathfrak{a}$ be an ideal of $R$ with $\dim(R/\mathfrak{a})\leq 2$, let $M$ be a finite $R$-module, let $X$ be an arbitrary $R$-module, and let $t$ be a non-negative integer such that $\operatorname{Ext}^{i}_{R}(M/\mathfrak{a}M, X)$ is a finite $R$-module for all $i\leq t+ 1$ and $\operatorname{Min}_R(\operatorname{Ext}^{i}_{R}(R/\mathfrak{a}, \operatorname{H}^j_{\mathfrak{a}}(M, X))/Y_{ij})_{= 0}$ is a finite set for all $i> 2$, for all $j< t$, and for all finite $R$-submodules $Y_{ij}$ of $\operatorname{Ext}^{i}_{R}(R/\mathfrak{a}, \operatorname{H}^j_{\mathfrak{a}}(M, X))$ $($e.g., $\operatorname{Supp}_R(M)\cap \operatorname{Supp}_R(X)\cap \operatorname{Var}(\mathfrak{a})\cap\operatorname{Max}(R)$ is a finite set$)$. Then the following statements are equivalent:
\begin{itemize}
\item[\emph{(i)}] $\operatorname{Hom}_{R}(R/\mathfrak{a}, \operatorname{H}^j_{\mathfrak{a}}(M, X))$ is a finite $R$-module for all $j\leq t$;
\item[\emph{(ii)}] $\operatorname{H}^j_{\mathfrak{a}}(M, X)$ is an $\mathfrak{a}$-cofinite $R$-module for all $j< t$.
\end{itemize}
\end{thm}

\begin{proof}
(i)$\Rightarrow$(ii). We prove by using induction on $t$. There is nothing to prove in the case that $t= 0$. Suppose that $t> 0$ and that $t- 1$ is settled. From the induction hypothesis, it is enough to show that $\operatorname{H}^{t- 1}_{\mathfrak{a}}(M, X)$ is an $\mathfrak{a}$-cofinite $R$-module. Since $\operatorname{H}^j_{\mathfrak{a}}(M, X)$ is an $\mathfrak{a}$-cofinite $R$-module for all $j< t- 1$ by the induction hypothesis, $\operatorname{Ext}^{1}_{R}(R/\mathfrak{a}, \operatorname{H}^{t- 1}_{\mathfrak{a}}(M, X))$ and $\operatorname{Ext}^{2}_{R}(R/\mathfrak{a}, \operatorname{H}^{t- 1}_{\mathfrak{a}}(M, X))$ are finite $R$-modules from our assumption and \cite[Corollary 2.2, Theorem 2.7, and Theorem 2.9]{VHH}. Therefore $\operatorname{H}^{t- 1}_{\mathfrak{a}}(M, X)$ is an $\mathfrak{a}$-cofinite $R$-module by \cite[Theorem 2.11]{VKhSh1}.

(ii)$\Rightarrow$(i). Follows from \cite[Corollary 2.2 and Theorem 2.3]{VHH}.
\end{proof}
%%% ----------------------------------------------------------------------

%%% ----------------------------------------------------------------------

%   {2.18}   ----------------------------------------------------------------------
\begin{thm}\label{2-18}
Let $\mathfrak{a}$ be an ideal of $R$ with $\dim(R/\mathfrak{a})\leq 2$, let $M$ be a finite $R$-module, and let $X$ be an arbitrary $R$-module such that $\operatorname{Ext}^{i}_{R}(M/\mathfrak{a}M, X)$ is a finite $R$-module for all $i$ and $\operatorname{Min}_R(\operatorname{Ext}^{i}_{R}(R/\mathfrak{a}, \operatorname{H}^j_{\mathfrak{a}}(M, X))/Y_{ij})_{= 0}$ is a finite set for all $i> 2$, for all $j$, and for all finite $R$-submodules $Y_{ij}$ of $\operatorname{Ext}^{i}_{R}(R/\mathfrak{a}, \operatorname{H}^j_{\mathfrak{a}}(M, X))$ $($e.g., $\operatorname{Supp}_R(M)\cap \operatorname{Supp}_R(X)\cap \operatorname{Var}(\mathfrak{a})\cap\operatorname{Max}(R)$ is a finite set$)$. Then the following statements hold true:
\begin{itemize}
\item[\emph{(i)}] $\operatorname{H}^j_{\mathfrak{a}}(M, X)$ is an $\mathfrak{a}$-cofinite $R$-module for all $j$ whenever $\operatorname{H}^{2j}_{\mathfrak{a}}(M, X)$ is an $\mathfrak{a}$-cofinite $R$-module for all $j$;
\item[\emph{(ii)}] $\operatorname{H}^j_{\mathfrak{a}}(M, X)$ is an $\mathfrak{a}$-cofinite $R$-module for all $j$ when $\operatorname{H}^{2j+ 1}_{\mathfrak{a}}(M, X)$ is an $\mathfrak{a}$-cofinite $R$-module for all $j$.
\end{itemize}
\end{thm}

\begin{proof}
(i). We prove that $\operatorname{H}^k_{\mathfrak{a}}(M, X)$ is an $\mathfrak{a}$-cofinite $R$-module for all $k\leq j$ by using induction on $j$. There is nothing to prove in the case that $j= 0$. Suppose that $j> 0$, $j- 1$ is settled, and $j$ is an odd integer. Thus $\operatorname{H}^k_{\mathfrak{a}}(M, X)$ is an $\mathfrak{a}$-cofinite $R$-module for all $k\leq j- 1$ from the induction hypothesis and $\operatorname{H}^{j+ 1}_{\mathfrak{a}}(M, X)$ is an $\mathfrak{a}$-cofinite $R$-module by our assumption. Hence, from \cite[Corollary 2.2 and Theorems 2.3, 2.7, and 2.9]{VHH}, $\operatorname{Ext}^{i}_{R}(R/\mathfrak{a}, \operatorname{H}^{j}_{\mathfrak{a}}(M, X))$ is a finite $R$-module for all $i\leq 2$. Therefore $\operatorname{H}^j_{\mathfrak{a}}(M, X)$ is an $\mathfrak{a}$-cofinite $R$-module by \cite[Theorem 2.11]{VKhSh1}.

(ii). The proof is similar to that of the first part and left to the reader.
\end{proof}
%%% ----------------------------------------------------------------------

%%% ----------------------------------------------------------------------

%   {2.19}   ----------------------------------------------------------------------
By taking $M=R$ in the above theorems, we have the following corollaries.

\begin{cor}\label{2-19}
Let $\mathfrak{a}$ be an ideal of $R$ with $\dim(R/\mathfrak{a})\leq 2$, let $X$ be an arbitrary $R$-module, and let $t$ be a non-negative integer such that $\operatorname{Ext}^{i}_{R}(R/\mathfrak{a}, X)$ is a finite $R$-module for all $i\leq t+ 1$ and $\operatorname{Min}_R(\operatorname{Ext}^{i}_{R}(R/\mathfrak{a}, \operatorname{H}^j_{\mathfrak{a}}(X))/Y_{ij})_{= 0}$ is a finite set for all $i> 2$, for all $j< t$, and for all finite $R$-submodules $Y_{ij}$ of $\operatorname{Ext}^{i}_{R}(R/\mathfrak{a}, \operatorname{H}^j_{\mathfrak{a}}(X))$ $($e.g., $\operatorname{Supp}_R(X)\cap \operatorname{Var}(\mathfrak{a})\cap\operatorname{Max}(R)$ is a finite set$)$. Then the following statements are equivalent:
\begin{itemize}
\item[\emph{(i)}] $\operatorname{Hom}_{R}(R/\mathfrak{a}, \operatorname{H}^j_{\mathfrak{a}}(X))$ is a finite $R$-module for all $j\leq t$;
\item[\emph{(ii)}] $\operatorname{H}^j_{\mathfrak{a}}(X)$ is an $\mathfrak{a}$-cofinite $R$-module for all $j< t$.
\end{itemize}
\end{cor}
%%% ----------------------------------------------------------------------

%%% ----------------------------------------------------------------------

%   {2.20}   ----------------------------------------------------------------------
\begin{cor}\label{2-20}
Let $\mathfrak{a}$ be an ideal of $R$ with $\dim(R/\mathfrak{a})\leq 2$ and let $X$ be an arbitrary $R$-module such that $\operatorname{Ext}^{i}_{R}(R/\mathfrak{a}, X)$ is a finite $R$-module for all $i$ and $\operatorname{Min}_R(\operatorname{Ext}^{i}_{R}(R/\mathfrak{a}, \operatorname{H}^j_{\mathfrak{a}}(X))/Y_{ij})_{= 0}$ is a finite set for all $i> 2$, for all $j$, and for all finite $R$-submodules $Y_{ij}$ of $\operatorname{Ext}^{i}_{R}(R/\mathfrak{a}, \operatorname{H}^j_{\mathfrak{a}}(X))$ $($e.g., $\operatorname{Supp}_R(X)\cap \operatorname{Var}(\mathfrak{a})\cap\operatorname{Max}(R)$ is a finite set$)$. Then the following statements hold true:
\begin{itemize}
\item[\emph{(i)}] $\operatorname{H}^j_{\mathfrak{a}}(X)$ is an $\mathfrak{a}$-cofinite $R$-module for all $j$ whenever $\operatorname{H}^{2j}_{\mathfrak{a}}(X)$ is an $\mathfrak{a}$-cofinite $R$-module for all $j$;
\item[\emph{(ii)}] $\operatorname{H}^j_{\mathfrak{a}}(X)$ is an $\mathfrak{a}$-cofinite $R$-module for all $j$ when $\operatorname{H}^{2j+ 1}_{\mathfrak{a}}(X)$ is an $\mathfrak{a}$-cofinite $R$-module for all $j$.
\end{itemize}
\end{cor}
%%% ----------------------------------------------------------------------
%%% ----------------------------------------------------------------------
%%% ----------------------------------------------------------------------
%\subsection*{Acknowledgements}
%%% ----------------------------------------------------------------------
%%% ----------------------------------------------------------------------
%%% ----------------------------------------------------------------------
%The authors would like to thank the referee for the invaluable comments on the manuscript.
%%% ----------------------------------------------------------------------
%%% ----------------------------------------------------------------------
%%% ----------------------------------------------------------------------
\bibliographystyle{amsplain}

\begin{thebibliography}{9}
\baselineskip=0.54cm
\bibitem{AbB}%001
N. Abazari and K. Bahmanpour,
\textit{Extension functors of local cohomology modules and Serre categories of modules},
Taiwanese J. Math. \textbf{19} (2015),
no. 1,
211--220.
%\href{https://mathscinet.ams.org/mathscinet-getitem?mr=MR3313413}{MR 3313413}.

\bibitem{AB}%002
M. Aghapournahr and K. Bahmanpour,
\textit{Cofiniteness of weakly Laskerian local cohomology modules},
Bull. Math. Soc. Sci. Math. Roumanie (N.S.) \textbf{57(105)} (2014),
no. 4,
347--356.
%\href{https://mathscinet.ams.org/mathscinet-getitem?mr=MR3288929}{MR 3288929}.

\bibitem{ATV}%003
M. Aghapournahr, A. J. Taherizadeh, and A. Vahidi,
\textit{Extension functors of local cohomology modules},
Bull. Iranian Math. Soc. \textbf{37} (2011),
no. 3,
117--134.
%\href{https://mathscinet.ams.org/mathscinet-getitem?mr=MR2901589}{MR 2901589}.

\bibitem{AN}%004
D. Asadollahi and R. Naghipour,
\textit{Faltings' local-global principle for the finiteness of local cohomology modules},
Comm. Algebra \textbf{43} (2015),
no. 3,
953--958.
%\href{https://mathscinet.ams.org/mathscinet-getitem?mr=MR3298115}{MR 3298115}.

\bibitem{B}%005
K. Bahmanpour,
\textit{On the category of weakly Laskerian cofinite modules},
Math. Scand. \textbf{115} (2014),
no. 1,
62--68.
%\href{https://mathscinet.ams.org/mathscinet-getitem?mr=MR3250048}{MR 3250048}.

\bibitem{BNS}%006
K. Bahmanpour, R. Naghipour, and M. Sedghi,
\textit{Cofiniteness with respect to ideals of small dimensions},
Algebr. Represent. Theory \textbf{18} (2015),
no. 2,
369--379.
%\href{https://mathscinet.ams.org/mathscinet-getitem?mr=MR3336344}{MR 3336344}.

\bibitem{BSh}%007
M. P. Brodmann and  R. Y. Sharp,
\textit{Local Cohomology: An Algebraic Introduction with Geometric Applications},
Cambridge Studies in Advanced Mathematics,
60,
Cambridge University Press,
Cambridge,
1998.
%\href{https://mathscinet.ams.org/mathscinet-getitem?mr=MR1613627}{MR 1613627}.

\bibitem{BH}%008
W. Bruns and J. Herzog,
\textit{Cohen-Macaulay Rings},
Cambridge Studies in Advanced Mathematics,
39,
Cambridge University Press,
Cambridge,
1993.
%\href{https://mathscinet.ams.org/mathscinet-getitem?mr=MR1251956}{MR 1251956}.

\bibitem{CGH}%009
N. T. Cuong, S. Goto, and N. V. Hoang,
\textit{On the cofiniteness of generalized local cohomology modules},
Kyoto J. Math. \textbf{55} (2015),
no. 1,
169--185.
%\href{https://mathscinet.ams.org/mathscinet-getitem?mr=MR3323531}{MR 3323531}.

\bibitem{CH}%010
N. T. Cuong and N. V. Hoang,
\textit{On the vanishing and the finiteness of supports of generalized local cohomology modules},
Manuscripta Math. \textbf{126} (2008),
no. 1,
59--72.
%\href{https://mathscinet.ams.org/mathscinet-getitem?mr=MR2395249}{MR 2395249}.

\bibitem{DM1}%011
K. Divaani-Aazar and A. Mafi,
\textit{Associated primes of local cohomology modules},
Proc. Amer. Math. Soc. \textbf{133} (2005),
no. 3,
655--660.
%\href{https://mathscinet.ams.org/mathscinet-getitem?mr=MR2113911}{MR 2113911}.

\bibitem{DM2}%012
K. Divaani-Aazar and A. Mafi,
\textit{Associated primes of local cohomology modules of weakly Laskerian modules},
Comm. Algebra \textbf{34} (2006),
no. 2,
681--690.
%\href{https://mathscinet.ams.org/mathscinet-getitem?mr=MR2211948}{MR 2211948}.

\bibitem{G}%013
A. Grothendieck,
Cohomologie locale des faisceaux coh$\acute{e}$rents et th$\acute{e}$or$\grave{e}$mes de Lefschetz locaux et globaux $(SGA$ 2$)$,
(French) Augment$\acute{e}$ $d'$un expos$\acute{e}$ par Mich$\grave{e}$le Raynaud,
S$\acute{e}$minaire de G$\acute{e}$om$\acute{e}$trie Alg$\acute{e}$brique du Bois-Marie,
1962,
Advanced Studies in Pure Mathematics,
Vol. 2,
\emph{North-Holland Publishing Co.},
\emph{Amsterdam; Masson $\&$ Cie, $\acute{E}$diteur, Paris},
1968.
%\href{https://mathscinet.ams.org/mathscinet-getitem?mr=MR0476737}{MR 0476737}.

\bibitem{Ha2}%014
R. Hartshorne,
\textit{Affine duality and cofiniteness},
Invent. Math. \textbf{9} (1969/1970),
no. 2,
145--164.
%\href{https://mathscinet.ams.org/mathscinet-getitem?mr=MR0257096}{MR 257096}.

\bibitem{HV}%015
S. H. Hassanzadeh and A. Vahidi,
\textit{On vanishing and cofiniteness of generalized local cohomology modules},
Comm. Algebra \textbf{37} (2009),
no. 7,
2290--2299.
%\href{https://mathscinet.ams.org/mathscinet-getitem?mr=MR2536919}{MR 2536919}.

\bibitem{He}%016
J. Herzog,
\textit{Komplexe, Aufl\"{o}sungen und Dualit\"{a}t in der lokalen Algebra},
Habilitationsschrift,
Universitat Regensburg,
1970.

\bibitem{Hu}%017
C. Huneke,
\textit{Problems on Local Cohomology, Free Resolutions in Commutative Algebra and Algebraic Geometry $($Sundance, UT, 1990$)$},
93--108,
Res. Notes Math.,
2,
Jones and Bartlett,
Boston,
MA,
1992.
%\href{https://mathscinet.ams.org/mathscinet-getitem?mr=MR1165320}{MR 1165320}.

\bibitem{Mel1}%018
L. Melkersson,
\textit{Modules cofinite with respect to an ideal},
J. Algebra \textbf{285} (2005),
no. 2,
649--668.
%\href{https://mathscinet.ams.org/mathscinet-getitem?mr=MR2125457}{MR 2125457}.

\bibitem{Mel}%019
L. Melkersson,
\textit{Cofiniteness with respect to ideals of dimension one},
J. Algebra \textbf{372} (2012),
%no. ,
459--462.
%\href{https://mathscinet.ams.org/mathscinet-getitem?mr=MR2990020}{MR 2990020}.

\bibitem{Rot}%020
J. J. Rotman,
\textit{An Introduction to Homological Algebra},
Second edition,
Universitext,
Springer,
New York,
2009.
%\href{https://mathscinet.ams.org/mathscinet-getitem?mr=MR2455920}{MR 2455920}.

\bibitem{Si}%021
A. K. Singh,
\textit{p-torsion elements in local cohomology modules},
Math. Res. Lett. \textbf{7} (2000),
no. 2--3,
165--176.
%\href{https://mathscinet.ams.org/mathscinet-getitem?mr=MR1764314}{MR 1764314}.

\bibitem{VA}%022
A. Vahidi and M. Aghapournahr,
\textit{Some results on generalized local cohomology modules},
Comm. Algebra \textbf{43} (2015),
no. 5,
2214--2230.
%\href{https://mathscinet.ams.org/mathscinet-getitem?mr=MR3316846}{MR 3316846}.

\bibitem{VHH}%023
A. Vahidi, F. Hassani, and E. Hoseinzade,
\textit{Finiteness of extension functors of generalized local cohomology modules},
Comm. Algebra \textbf{47} (2019),
no. 3,
1376--1384.
%\href{https://mathscinet.ams.org/mathscinet-getitem?mr=MR3938561}{MR 3938561}.

\bibitem{VKhSh1}%024
A. Vahidi, A. Khaksari, and M. Shirazipour,
\textit{Cofinite modules and cofiniteness of local cohomology modules},
Rev. Un. Mat. Argentina,
%no. ...,
%...--....
%\href{}{}.
(in press).

\bibitem{VM}%025
A. Vahidi and S. Morsali,
\textit{Cofiniteness with respect to the class of modules in dimension less than a fixed integer},
Taiwanese J. Math. \textbf{24} (2020),
no. 4,
825--840.
%\href{https://mathscinet.ams.org/mathscinet-getitem?mr=MR4124548}{MR 4124548}.

\bibitem{VM2}%026
A. Vahidi and S. Morsali,
\textit{Cofiniteness of generalized local cohomology modules with respect to the class of modules in dimension less than a fixed integer},
Bull. Belg. Math. Soc. Simon Stevin \textbf{28} (2022),
no. 5,
615--631.
%\href{...}{...}.

\bibitem{VP}%027
A. Vahidi and M. Papari-Zarei,
\textit{Cofiniteness of local cohomology modules in the class of modules in dimension less than a fixed integer},
Rev. Un. Mat. Argentina \textbf{62} (2021),
no. 1,
191--198.
%\href{https://mathscinet.ams.org/mathscinet-getitem?mr=MR4281080}{MR 4281080}.

\bibitem{VP2}%028
A. Vahidi and M. Papari-Zarei,
\textit{On the cofiniteness of generalized local cohomology modules with respect to the class of modules in dimension less than a fixed integer},
Comm. Algebra \textbf{49} (2021),
no. 8,
3423--3431.
%\href{https://mathscinet.ams.org/mathscinet-getitem?mr=MR4283158}{MR 4283158}.

\bibitem{Ya}%029
S. Yassemi,
\textit{Cofinite modules},
Comm. Algebra \textbf{29} (2001),
no. 6,
2333--2340.
%\href{https://mathscinet.ams.org/mathscinet-getitem?mr=MR1845114}{MR 1845114}.

\bibitem{Yo}%030
T. Yoshizawa,
\textit{Subcategories of extension modules by Serre subcategories},
Proc. Amer. Math. Soc. \textbf{140} (2012),
no. 7,
2293--2305.
%\href{https://mathscinet.ams.org/mathscinet-getitem?mr=MR2898693}{MR 2898693}.
\baselineskip=0.55cm
\end{thebibliography}
%%% ----------------------------------------------------------------------
%%% ----------------------------------------------------------------------
%%% ----------------------------------------------------------------------

%%% ----------------------------------------------------------------------
%%% ----------------------------------------------------------------------
%%% ----------------------------------------------------------------------
\vspace*{5mm}
\author{
{\bf Alireza Vahidi, Ahmad Khaksari, and Mohammad Shirazipour}\\
\small Department of Mathematics, Payame Noor University, Tehran, Iran\\
\small E-mail address: vahidi.ar@pnu.ac.ir, a\underline{ }khaksari@pnu.ac.ir, and m\underline{ }shirazipour@pnu.ac.ir
}
%%% ----------------------------------------------------------------------
%%% ----------------------------------------------------------------------
%%% ----------------------------------------------------------------------
\end{document}